\documentclass{amsart}

      \usepackage{amssymb}
      \usepackage{verbatim}

      \theoremstyle{plain}
      \newtheorem{theorem}{Theorem}[section]
      \newtheorem{lemma}[theorem]{Lemma}

      \theoremstyle{definition}
      \newtheorem{definition}[theorem]{Definition}

      \theoremstyle{remark}

      \newcommand{\R}{{\mathbb R}}

      \makeatletter
      \def\@setcopyright{}
      \def\serieslogo@{}
      \makeatother

\begin{document}

   \author{Fei He}
   \address{School of Mathematics, University of Minnesota, Twin Cities}
   \email{hexxx221@umn.edu}

   \title[]{Some rigidity results for noncompact gradient steady Ricci solitons and Ricci-flat manifolds}

   \begin{abstract}
     Gradient steady Ricci solitons are natural generalizations of Ricci-flat manifolds. In this article, we prove a curvature gap theorem for gradient steady Ricci solitons with nonconstant potential functions; and a curvature gap theorem for Ricci-flat manifolds, removing the volume growth assumptions in known results.
   \end{abstract}



   \thanks{This research was partially supported by NSF grant DMS-1109342 (P.I.: Chuu-Lian Terng).}
   \thanks{}


   \date{\today}

   \maketitle

   \section{Introduction}
A gradient Ricci soliton is a triple $(M, g, f)$, where $g$ is a Riemannian metric on a smooth manifold $M$, and $f$ is a smooth function satisfying
\[ Ric(g)+ Hess(f)= \lambda g.
\]
It is called shrinking, steady or expanding when $\lambda $ is positive, $0$ or negative respectively. As the name suggested, Ricci solitons came from the study of Ricci flows, as self-similar solutions and important singularity models. Meanwhile, they are natural generalizations of Einstein manifolds, i.e. Riemannian manifolds $(M,g)$ satisfying
 \[ Ric(g)=\lambda g,
 \]
 for some constant $\lambda$, this is the viewpoint that we take in this article.

We will study gradient steady Ricci solitons, including Ricci-flat manifolds as a special case. Indeed, any Ricci-flat manifold can be viewed as a gradient steady Ricci soliton with constant potential function. However, the potential function for a Ricci-flat steady Ricci soliton is not necessarily constant, for example, $(\R^n,g_{\scriptscriptstyle E}, x_1)$, where $x_1$ is the first coordinate function. We call a gradient steady Ricci soliton \textit{nontrivial} if its potential function is nonconstant. Since it is known that compact steady Ricci solitons are Ricci-flat, our focus here is on the noncompact case.

Gradient Ricci solitons and Einstein manifolds are very rigid because their Riemann curvature tensors satisfy elliptic systems. We will not survey the various rigidity results, instead, we only recall those closely related to ours. For Ricci-flat manifolds with maximal volume growth, a curvature gap theorem has been implied by the work of Bando, Kasue and Nakajima~\cite{Bando-Kasue-Nakajima1989}, and independently proved by Shen in \cite{Shen1990}, the proof depends on a Euclidean type Sobolev inequality. Minerbe generalized this gap theorem to Ricci-flat manifolds with much weaker volume growth assumptions in~\cite{Minerbe2009}, by establishing weighted Sobolev and Poincare inequalities. Similar results with no volume growth assumption have been obtained by Carron (\cite{Carron2010}), in the more general setting of critical metrics, the proof only relies on local regularity estimates of the Riemann curvature tensor. Our first goal is to generalize these results to noncompact gradient steady Ricci solitons. We prove the following:

\begin{theorem}
\label{gapsoliton}
There exist constants $c(n)$ and $\epsilon(n)$ depending only on $n$, such that for any complete noncompact gradient steady Ricci soliton $(M,g,f)$ with dimension $n \geq 3$, and $R+|\nabla f|^2=\Lambda>0$ where $R$ is the scalar curvature, if
\[ \int_M |Rm|^\frac{n}{2}r^nV_f(r)^{-1}e^{c(n)\sqrt{\Lambda}r}dv_f < \epsilon(n)e^{-f(p)},
\]
where $r(x)=dist(p,x)$ for some $p \in M$, $V_f(r)=\int_{B_p(r)}dv_f$, $dv_f=e^{-f}dv$, then $(M,g)$ is flat, and the pullback of $f$ to the universal cover is a linear function.
\end{theorem}

Recall the cigar soliton discovered by Hamilton (\cite{Hamilton1986}):
\[\left(\mathbb{R}^2, ds^2=\frac{dx^2+dy^2}{1+x^2+y^2}, f=-\log(1+x^2+y^2)\right).\]
Its product with any compact flat manifold is a non-flat gradient steady Ricci soliton with exponential curvature decay, therefore the $L^\frac{n}{2}$ integral of its sectional curvature, weighted as in Theorem \ref{gapsoliton} with probably smaller $c(n)$, is finite.

In the following, the term Sobolev inequality will be used to refer to Sobolev inequalities possibly with an $L^2$ term on the right hand side, i.e. inequalities of the form:
\[\left(\int \phi^\frac{2n}{n-2}\right)^\frac{n-2}{n}\leq C \left(\int |\nabla \phi|^2+\int\phi^2\right).\]
And the term Euclidean type Sobolev inequality refers to inequalities of the form:
\[\left(\int \phi^\frac{2n}{n-2}\right)^\frac{n-2}{n}\leq C \int |\nabla \phi|^2.\]
We will prove Theorem \ref{gapsoliton} by a weighted Sobolev inequality of Euclidean type. Recall that Riemannian manifolds with nonnegative Ricci curvature has a uniform volume doubling constant, and the validity of Euclidean type Sobolev ineqaulities is equivalent to maximal volume growth. If volume growth is only super-quadratic, Minerbe showed there are weighted Euclidean type Sobolev inequaltities (\cite{Minerbe2009}), the proof of which also depends sensitively on the volume growth assumption and the uniform volume doubling constant. However, in the case of gradient steady Ricci solitons, there does not exist a uniform volume doubling constant for geodesic balls of all sizes. On the other hand, Munteanu and Wang showed in~\cite{Munteanu-Wang2011} that nontrivial gradient steady Ricci solitons have positive $f$-spectrum. We can apply this result to improve a Sobolev inequality with $L^2$ term into Euclidean type, provided we use the weighted measure $e^{-f}dv$. Therefore we only need to establish a weighted Sobolev inequality with $L^2$ term, which will be proved in a more general setting.

Also, as shown in \cite{Minerbe2009}, (weighted) Poincare inequalities directly imply curvature gap theorems. So Munteanu and Wang's result can be used to prove the following:
\begin{theorem} \label{gap by poincare}
There exists a constant $C(n)$, such that for any complete noncompact gradient steady Ricci soliton $(M,g,f)$ with dimension $n \geq 3$, and $R+|\nabla f|^2=\Lambda >0$, if for some $p\in M$ and $\alpha \geq 1$,
\[\int_{B_p(2R)\backslash B_p(R)}  |Rm|^{2\alpha} dv_f = o(R^2) \text{ as $R\to \infty$},\]
and
\[\sup_M|Rm|< \frac{\Lambda}{4\alpha C(n)},\]
then $(M,g)$ is flat, and the pullback of $f$ to the universal cover is a linear function.
\end{theorem}

We would like to remark that since $|\nabla f|$ is bounded, the proof in \cite{Carron2010} with little modification shows that gradient steady Ricci solitons are $(\Gamma, k)$ regular in the definition of Carron (see section \ref{Ricci-flat manifolds}). Therefore, some results in \cite{Carron2010} which only depend on the local regularity property still hold for gradient steady Ricci solitons, these results imply that the curvature cannot decay too fast locally.

Our method for nontrivial gradient steady Ricci solitons doesn't apply to Ricci-flat manifolds since the later necessarily have zero bottom spectrum. However, we can apply local Dirichlet type Sobolev and Poincare inequalities to prove some rigidity results, which roughly tell that the curvature cannot decay too fast in integral sense relative to its $L^\infty$ or $L^\frac{n}{2}$ norm. Moreover, we can combine local Sobolev inequalities and the regularity of Ricci-flat metrics to prove the following:

\begin{theorem}\label{gap theorem for ricci-flat manifolds}
There exists an $\epsilon$ depending only on $n$, such that for any complete noncompact Ricci-flat Riemannian manifold $(M,g)$ of dimension $n$, if there exists a point $p\in M$ such that
\[\int_M |Rm|^\frac{n}{2} \rho_p^{-1}(r) < \epsilon,\]
where
\[ \rho_p(r)=\frac{V(B_p(r))}{r^n},\]
then $Rm\equiv 0$.
\end{theorem}
This result has been proved in \cite{Minerbe2009} with an additional assumption that the volume growth is super-quadratic, which is removed here. One major motivation for considering Ricci-flat manifolds with small volume growth is the existence of a large class of asymptotically cylindrical Ricci-flat manifolds, which has been studied in \cite{Tian-yau1990,Tian-Yau1991,Kovalev2003,Haskins-Hein-Nordstrom2012,Corti-Alessio-Haskins-Nordstrom-Pacini2013}. These manifolds have exactly linear volume growth.
%
Examples of non-flat ACyl Ricci-flat manifolds asymptotic to flat metrics at infinity are given by \cite{Haskins-Hein-Nordstrom2012}. Such manifolds have their sectional curvature integrable in $L^\frac{n}{2}$ norm weighted by $\rho_p^{-1}$ as in Theorem \ref{gap theorem for ricci-flat manifolds}. Moreover, their products with scaled circles $\epsilon\mathbb{S}^1$ have arbitrarily small $L^\frac{n}{2}$ integral of curvature as $\epsilon \to 0$, justifying the volume ratio term in Theorem \ref{gap theorem for ricci-flat manifolds}.

The organization of this paper is as the following: In the relatively independent section \ref{weightedSI} we prove weighted Sobolev inequalities on smooth metric measure spaces which are more general than gradient steady Ricci solitons, and may be of independent interests. We prove Theorem \ref{gapsoliton} in section \ref{steady ricci soliton}. Ricci-flat manifolds will be discussed in section \ref{Ricci-flat manifolds}.

\textbf{Acknowledgements:} This work was done when the author was studying at the University of California, Irvine, and the results have been included in the author's thesis \cite{he2014regularity}. The author is deeply grateful to his advisor professor Peter Li. He would also like to thank professor Jeffrey Streets for helpful discussions, and professor Chuu-Lian Terng for her generous support.

   \section{Weighted Sobolev inequalities}
   \label{weightedSI}
In this section, we let $(M,g,d\lambda)$ be a weighted manifold with dimension $n\geq 3$, where $d\lambda=wdv$ for some positive smooth function $w$ on $M$. We can write the volume form in polar coordinates as
\[ dv(exp_x(r,\theta))=J(x,r,\theta)drd\theta,
\]
and let $J_\lambda(x,r,\theta)=wJ(x,r,\theta)$. We say that $(M,g,d\lambda)$ satisfies the \textit{exponential Jacobian comparison property} with respect to $p\in M$ if there is a constant $c_0$, such that
\begin{equation}\label{jacobian comparison}
\frac{J_\lambda(x,r_2,\theta)}{J_\lambda(x,r_1,\theta)}\leq e^{c_0 R}\left(\frac{r_2}{r_1}\right)^{n-1},
\end{equation}
for any $R>0$, $x\in B_p(R)$ and $0<r_1<r_2\leq R$.

Condition (\ref{jacobian comparison}) immediately implies the \textit{exponential volume comparison property} that
\begin{equation}\label{volume comparison}
\frac{V_\lambda(x,r_2)}{V_\lambda(x,r_1)}\leq e^{c_0 R}\left(\frac{r_2}{r_1}\right)^{n},
\end{equation}
for any $R>0$, $x\in B_p(R)$ and $0<r_1<r_2\leq R$, where
\[ V_\lambda(x,r)=\int_{B_x(r)}d\lambda .
\]
Typical examples satisfying (\ref{jacobian comparison}) are smooth metric measure spaces $(M,g,e^{-f}dv)$ with linear potential function and nonnegative $\infty$-Bakry-Emery Ricci tensor
\[ Ric(g)+Hess(f)\geq 0,
\]
which include gradient steady Ricci solitons as special cases.

The goal of this chapter is to prove a weighted Sobolev inequality for smooth metric measure spaces satisfying (\ref{jacobian comparison}), see Theorem \ref{thm: weighted sobolev inequality for metric measure space}. For gradient steady Ricci solitons, the potential function has bounded gradient, hence the volume comparison constant is uniform for all unit geodesic balls (see \cite{Wei-Wylie2009}). Thus we can directly apply well-known results to obtain $L^2$-Sobolev inequalities for each unit ball, then glue them by a partition of unity. This is actually sufficient for our application in this work.

However, in general the volume comparison constant may have exponential decay even for small regions. Since it may be of independent interests, we prove Theorem \ref{thm: weighted sobolev inequality for metric measure space} in the general case. And given that Euclidean type Sobolev inequalities are more useful, the proof is presented in a manner that we keep the local Sobolev inequalities Euclidean type as long as we can, and the $L^2$-term is only introduced in the last step when we glue them by a partition of unity.

  We can use (1) and the method of Buser (\cite{Buser1982}) to prove Neumann Poincare inequalities on geodesic balls:
\begin{lemma}\label{NPI on balls}
There exist $C_1$ and $C_2$ depending only on $n$ and $c_0$, such that for any $R>0$, $x\in B_p(R)$, $0<r \leq R$ and $\phi \in C^1(B_x(r))$, we have
\[ \int_{B_x(r)}|\phi-\phi_{B_x(r)}|^2 d\lambda \leq C_1 e^{C_2 R} r^2 \int_{B_x(r)}|\nabla \phi|^2 d\lambda,
\]
where $\phi_{B_x(r)}=V_\lambda(x,r)^{-1}\int_{B_x(r)}\phi d\lambda$.
\end{lemma}
The proof is omitted here, one can refer to \cite{Munteanu-Wang2011} for a detailed proof in the smooth metric measure space setting.
Using (\ref{volume comparison}) and Lemma \ref{NPI on balls}, we can apply the method of Maheux-Sallof-Coste (\cite{Maheux-Saloff-Coste1995}) to prove Neumann Sobolev inequalities on balls:

\begin{lemma}\label{NSI on balls}
There exist $C_1$ and $C_2$ depending only on $n$ and $c_0$, such that for any $R>0$, $x\in B_p(R)$, $0<r \leq R$ and $\phi \in C^1(B_x(r))$, we have
\[ \left( \int_{B_x(r)}|\phi-\phi_{B_x(r)}|^{\frac{2n}{n-2}} d\lambda \right)^{\frac{n-2}{n}} \leq C_1 e^{C_2 R} \frac{r^2}{V_\lambda(x,r)^{2/n}} \int_{B_x(r)}|\nabla \phi|^2 d\lambda,
\]
where $\phi_{B_x(r)}=V_\lambda(x,r)^{-1}\int_{B_x(r)}\phi d\lambda$.
\end{lemma}

\begin{proof}
The proof is the same as in \cite{Maheux-Saloff-Coste1995}. However, the Sobolev constants depend implicitly on the volume doubling constant in the original statement in \cite{Maheux-Saloff-Coste1995}, we need to make it explicit, so we have to compute the constants more carefully. We sketch it here, and refer to \cite{Jerison1986,Maheux-Saloff-Coste1995,Bakry-Coulhon-Ledoux-Saloff-Coste1995} for details.

Denote $D$ to be the volume doubling constant of $B_p(R)$, then $D=2^ne^{c_0R}$. Our goal is to write the Sobolev constant in the form of $c_1(n)D^{c_2(n)}$, where $c_1(n)$ and $c_2(n)$ are some constants depending on $n$. In \cite{Maheux-Saloff-Coste1995}, a Whitney cover $\mathcal{F}_s$ of $B_p(R)$ is constructed for each $0<s\leq R$. By the construction, $\mathcal{F}_s$ is a disjoint union of $\mathcal{F}_{0,s}$ and $\mathcal{F}_{1,s}$, for each $B\in \mathcal{F}_{0,s}$, it's centered in $B_p(R-0.999s)$ with radius $r(B)=10^{-3}s$; for each $B\in \mathcal{F}_{1,s}$, $r(B)=10^{-3}dist(B,\partial B_p(R))$. And the number of $10B$ covering a point in $B_p(R)$ cannot exceed $K=C(n)D$.

For each $B$ in $\mathcal{F}_{1,s}$, there are finitely many balls in $\mathcal{F}_s$ denoted by $A_i$, $i=0,1,...,l(B)$, such that
\[A_0=B; 2A_i \cap 2A_{i+1}\neq \emptyset; A_i\in \mathcal{F}_{1,s} \text{ for } i=0,1,...,l(B)-1; A_{l(B)} \in \mathcal{F}_{0,s};
\]
and $A_i\cap \gamma_B \neq \emptyset$, where $\gamma_B$ is a minimal geodesic joining $p$ and the center of $B$. For each $B\in \mathcal{F}_s$, define
\[ \mathcal{F}_s(B)=\{A_i| i=0,1,...,l(B)\};
\]
and for each $A \in \mathcal{F}_s$, define
\[ A(\mathcal{F}_s)=\{B\in \mathcal{F}_s| A \in \mathcal{F}_s(B)\}.
\]
It is estimated in \cite{Jerison1986} that
\[ \sharp \mathcal{F}_s(B) \leq C(n)D \log\left(\frac{s}{r(B)}\right),
\]
and
\[ \sum_{B\in A(\mathcal{F}_s), r\leq r(B) \leq 2r} V_\lambda(B)\leq C_1(n)D^{C_2(n)}\left(\frac{r}{r(A)}\right)^\epsilon V_\lambda(A),
\]
where we can take
\[\epsilon= \log_4 (1+C(n)D^{-1}).\]
We have to estimate the following quantity, take any $p>1$,
\[ \begin{split}
 &\sum_{B\in A(\mathcal{F}_s)} \sharp \mathcal{F}_s(B)^{p-1} V_\lambda(B)\\
  \leq & \sum_{i=1}^\infty \sum_{(\frac{1}{2})^i 10^2 r(A) \leq r(B) \leq (\frac{1}{2})^{i-1} 10^2 r(A)}\sharp \mathcal{F}_s(B)^{p-1} V_\lambda(B) \\
\leq &\sum_{i=1}^\infty \left(C(n)D \log \frac{s}{(1/2)^i10^2r(A)}\right)^{p-1} C_1(n)D^{C_2(n)} \left(\frac{(1/2)^i 10^2 r(A)}{r(A)}\right)^\epsilon V_\lambda(A) \\
\leq &\sum_{i=1}^\infty C(n,p) D^{C(n,p)} \left(\log{\frac{s}{r(A)}}-\log{10^2(\frac{1}{2})^i}\right)^{p-1} \left(10^2 (\frac{1}{2})^i\right)^\epsilon V_\lambda(A).
\end{split}
\]
Denote $l_i=10^2(\frac{1}{2})^i$.
\[ \begin{split}
 & \left(\log{\frac{s}{r(A)}}-\log{l_i}\right)^{p-1}l_i^\epsilon \\
 = &\left(\frac{p-1}{\epsilon}\right)^{p-1} \left(\frac{\epsilon}{p-1}l_i^{\frac{\epsilon}{p-1}}\log{\frac{s}{r(A)}}-l_i^{\frac{\epsilon}{p-1}}\log{l_i^{\frac{\epsilon}{p-1}}}\right)^{p-1} \\
 \leq & 2^{p-1}\left(\frac{p-1}{\epsilon}\right)^{p-1} \left(\left(\frac{\epsilon}{p-1}\right)^{p-1}l_i^{\epsilon}\left(\log{\frac{s}{r(A)}}\right)^{p-1}+2^{p-1}e^{1-p}l_i^{\epsilon/2}\right).\\
\end{split}
\]
Since
\[ \sum_{i=1}^\infty l_i^\epsilon = 10^2 \sum _{i=1}^\infty (\frac{1}{2^\epsilon})^i \leq 10^2 (\ln2) \epsilon^{-1},
\]
we get
\[ \sum_{B\in A(\mathcal{F}_s)} \sharp \mathcal{F}_s(B)^{p-1} V_\lambda(B) \leq C(n,p)D^{C(n,p)}  \max\{\epsilon^{-1}, \epsilon^{-p}\} \left(  \log{\frac{s}{r(A)}}\right)^{p-1}V_\lambda(A).
\]
Note that $\ln(1+x)> x^2$ when $0<x<1/4$, so we get from the above estimates:
\[  \sum_{B\in A(\mathcal{F}_s)} \sharp \mathcal{F}_s(B)^{p-1} V_\lambda(B) \leq C(n,p)D^{C(n,p)} \left(  \log{\frac{s}{r(A)}}\right)^{p-1}V_\lambda(A),
\]
which is Lemma 4.3 in \cite{Maheux-Saloff-Coste1995}. The dependency of constants is then clear in the rest of the proof, see \cite{Maheux-Saloff-Coste1995,Bakry-Coulhon-Ledoux-Saloff-Coste1995}.

Note we only need to take $p=2$ in this lemma.
\end{proof}

With this lemma, we can prove weak Dirichlet Sobolev inequalities on connected components of annuli by using the method in \cite{Grigor'yan-Saloff-Coste2005}.
In the following we let
\[E_{-\delta}:= \{x | d(x,M-E) > \delta (R_2-R_1) \},\]
and let $E_\delta$ be the $\delta-$neighbourhood of $E$.
\begin{lemma}\label{DSI on annuli}
There exist constants $C_1(n,\delta)$ and $C_2(n,c_0)$, for any connected component $E$ of annulus $A(R_1,R_2):= B_p(R_2)\backslash B_p(R_1)$, $R_2>R_1>0$, and any $\delta <<1$, we have
\[ \left( \int_{E_{-\delta}} |\phi|^{\frac{2n}{n-2}} d\lambda \right)^\frac{n -2}{n} \leq C_1e^{C_2 R_2} \frac{R_2^2}{V_\lambda(B_p(R_2))^{2/n}} \int_{E_\delta}|\nabla \phi|^2 d\lambda,
\]
for any $\phi \in C_0^1(E_{-\delta})$.
\end{lemma}

\begin{proof}
Let $E$ be a connected component of  $A(R_1,R_2)$. Choose a $\delta$-lattice of $E$, i.e. a maximal set $(x_i)_{i \in I}$ in $E$ such that $d(x_i,x_j)\geq \delta(R_2-R_1)/3$ when $i \neq j$. For convenience, let $r=\delta (R_2-R_1)/6$, and let $\tilde{B_i}, B_i, \hat{B_i}$ denote $B_{x_i}(r), B_{x_i}(2r), B_{x_i}(6r)$ respectively. Then $\tilde{B_i} \cap \tilde{B_j} = \varnothing$ whenever $i \neq j$, and $(B_i)_{ i \in I}$ is a finite cover for $E$. By (\ref{volume comparison}), we can estimate
\[ \sharp I \leq e^{2c_0R_2} \left(\frac{2R_2}{r}\right)^n.
\]
Let $N_I$ be the maximal intersection number, i.e. any $B_i$ can intersect nontrivially with at most $N_I$ balls in the family $(B_i)_{ i\in I}$. Then
\[ N_I \leq e^{c_0 R_2}6^n.
\]
Similarly we can define and estimate
\[ \hat{N_I} \leq e^{c_0 R_2} 18^n.
\]
We also need to compare the volume of different balls in this family, let
\[ C_V= \sup_{i,j \in I} \frac{V_\lambda(x_i,2r)}{V_\lambda(x_j,2r)},
\]
then
\[C_V \leq e^{2c_0 R_2}\left(\frac{R_2}{r}\right)^n.
\]

For any $\phi \in C_0^1(E_{-\delta})$,
\[ \begin{split}
         & \int_E |\phi|^{\frac{2n}{n-2}} d\lambda \\
    \leq & \sum_I \int_{B_i} |\phi|^{\frac{2n}{n-2}} d\lambda \\
    \leq & 2^{\frac{n+2}{n-2}}\sum_I \int_{B_i} |\phi-\phi_{B_i}|^{\frac{2n}{n-2}} d\lambda + 2^{\frac{n+2}{n-2}}\sum_I \int_{B_i} \phi_{B_i}^{\frac{2n}{n-2}} d\lambda \\
    =    & J_1 + J_2.
   \end{split}
\]
By Lemma \ref{NSI on balls},
\[ \begin{split}
      J_1  & \leq 2^{\frac{n+2}{n-2}} \sum_I \left( C_1 e^{C_2 R_2}\frac{(2r)^2}{V_\lambda(B_i)^{2/n}} \int_{B_i} |\nabla \phi|^2 d\lambda \right)^{\frac{n}{n-2}} \\
           & \leq 2^{\frac{n+2}{n-2}}N_I \left( C_1e^{(4c_0/n+C_2)R_2} \frac{R_2^2}{V_\lambda(B_p(R_2))^{2/n}}\int_{E_\delta} |\nabla \phi|^2d\lambda \right)^{\frac{n}{n-2}},
   \end{split}
\]
where $E_\delta$ is the $\delta(R_2-R_1)$-neighbourhood of $E$.
To estimate $J_2$, we need a lemma from \cite{Minerbe2009}:
\begin{lemma}[V. Minerbe]\label{discrete PI on finite graph}
Let $\mathcal{G}=(\mathcal{V},\mathcal{E})$ be a finite graph, $\mathcal{V}$ is the set of vertices and $\mathcal{E}$ is the set of edges. Fix $p\geq 1$, then for any function $f$ on $\mathcal{V}$ which vanishes at some points,
\[ \sum_{i\in \mathcal{V}} |f(i)|^p \leq (\sharp \mathcal{V})^p \sum_{(i,j)\in \mathcal{E}} |f(i)-f(j)|^p.
\]
\end{lemma}

In our setting, we define a graph by letting $\mathcal{V}=I$ and $(i,j)\in \mathcal{E}$ if and only if $B_i \cap B_j \neq \emptyset$. Let's denote $\phi_i=\phi_{B_i}$.
\[ \begin{split}
     2^{-\frac{n+2}{n-2}}J_2  =  & \sum_I \phi_i^{\frac{2n}{n-2}} V_\lambda(B_i) \\
                            \leq & \max_{i \in I} V_\lambda(B_i) \sum_I \phi_i^{\frac{2n}{n-2}} \\
                            \leq & \max_{i \in I} V_\lambda(B_i) (\sharp I)^{\frac{2n}{n-2}} \sum_\mathcal{E} |\phi_i- \phi_j|^{\frac{2n}{n-2}} \\
                            \leq & C_V (\sharp I)^{\frac{2n}{n-2}} \sum_\mathcal{E} |\phi_i- \phi_j|^{\frac{2n}{n-2}} \max{(V_\lambda(B_i),V_\lambda(B_j))},
   \end{split}
\]
where we used Lemma \ref{discrete PI on finite graph} in the second inequality.
\[ \begin{split}
          & \sum_\mathcal{E} |\phi_i- \phi_j|^{\frac{2n}{n-2}} \max{(V_\lambda(B_i),V_\lambda(B_j))}\\
        = & \sum_\mathcal{E} \frac{\max(V_\lambda(B_i),V_\lambda(B_j))}{V_\lambda(B_i)^{2n/(n-2)}V_\lambda(B_j)^{2n/(n-2)}} \left| \int_{B_i}\int_{B_j} (\phi(x)-\phi(y)) d\lambda(x)d\lambda(y) \right|^{\frac{2n}{n-2}} \\
     \leq & \sum_\mathcal{E} \frac{\max(V_\lambda(B_i),V_\lambda(B_j))}{V_\lambda(B_i)V_\lambda(B_j)} \int_{B_i}\int_{B_j} |\phi(x)-\phi(y)|^{\frac{2n}{n-2}}  d\lambda(x)d\lambda(y) \\
     \leq & \sum_\mathcal{E} \frac{1}{V_\lambda(B_i)} \int_{\hat{B}_i}\int_{\hat{B}_i} |\phi(x)-\phi(y)|^{\frac{2n}{n-2}}  d\lambda(x)d\lambda(y) \\
     \leq & 2^{\frac{n+2}{n-2}} N_I \sum_I \frac{V_\lambda(\hat{B_i})}{V_\lambda(B_i)}\int_{\hat{B}_i}|\phi-\phi_{\hat{B}_i}|^{\frac{2n}{n-2}} d\lambda,
   \end{split}
\]
the second last inequality comes from the observation that $B_j\subset \hat{B_i}$ when $B_i\cap B_j\neq \emptyset$. Apply volume comparison (\ref{volume comparison}) and Lemma \ref{NSI on balls}, we get
\[ \begin{split}
      J_2 \leq & 2^{\frac{2(n+2)}{n-2}} N_I C_V (\sharp I)^{\frac{2n}{n-2}}e^{c_0 R_2} 3^n\sum_I \left(C_1 e^{ C_2R_2} \frac{(6r)^2}{V_\lambda(\hat{B_i})^{2/n}} \int_{\hat{B}_i}|\nabla \phi|^2 d\lambda \right)^{\frac{n}{n-2}} \\
          \leq & 2^{\frac{2(n+2)}{n-2}} N_I \hat{N_I} C_V (\sharp I)^{\frac{2n}{n-2}}e^{c_0 R_2} 3^n \left( C_1 e^{ (4c_0/n+C_2)R_2} \frac{R_2^2}{V_\lambda(B_p(R_2))^{2/n}} \int_{E_\delta}|\nabla \phi|^2 d\lambda \right)^{\frac{n}{n-2}}.
   \end{split}
\]

Using the above estimates, we can finish the proof of Lemma \ref{DSI on annuli}.
\end{proof}

Now we can glue up local sobolev inequalities to get a global weighted sobolev inequality.
\begin{theorem}\label{thm: weighted sobolev inequality for metric measure space}
Let $(M,g,d\lambda)$ be a smooth weighted Riemannian manifold satisfying property (\ref{jacobian comparison}), then there exist constants $\alpha(n,c_0)$ and $C_S(n)$, such that for any $\phi \in C_0^1(M)$,
\[ \left(\int_M |\phi|^\frac{2n}{n-2}\left(\frac{V_\lambda(r)}{r^n}\right)^\frac{2}{n-2}e^{-\alpha r} d\lambda\right)^\frac{n-2}{n}\leq C_S \int_M \phi^2 + |\nabla \phi|^2 d\lambda.
\]
\end{theorem}

\begin{proof}
First we decompose $M$ into connected components of annuli.
Choose
\[ R_0=0, R_i=2^i, i=1,2,3,...,
\]
let
\[ A_i=B_p(R_{i+1})-B_p(R_i), i=0,1,2,...
\]
and denote the connected components of $A_i$ as
\[ E_{i1}, E_{i2},...,E_{il_i}.
\]
If $E_{ij}$ is not connected to $A_{i+1}$, then it must be connected to $E_{i-1,k}$ for some $k$, in this case, we delete $E_{ij}$ from the list and merge it to $E_{i-1,k}$. Similarly, if $E_{ij}$ is not connected to $A_{i-1}$, we merge it to some connected component of $A_{i+1}$.

Finally we let $U_{ij}=E_{ij}\cup E_{i-1,k}\cup E_{i+1,l}$ where the union is for all $E_{i-1,k}$ and $E_{i+1,l}$ connected to $E_{ij}$ in $A(R_{i-1}, R_{i+2})$.

Since every $E_{ij}$ contains a ball of radius $R_{i-1}$, we can use the volume comparison property (\ref{volume comparison}) to estimate
\[ l_i \leq e^{c_0 R_{i+3}} 16^n.
\]

We choose smooth nonnegative cut-off functions $\psi_{ij}=1$ on $E_{ij}$ and $\psi_{ij}=0$ on $M\backslash (U_{ij})_{-\delta}$, and define a partition of unity
\[ \eta_{ij}= \frac{\psi_{ij}}{\sum \psi_{ij}}.
\]
$\eta_{ij}$ is well-defined since the sum is finite and positive at every point on $M$. We can choose $\psi_{ij}$ properly so that $|\nabla \psi_{ij}|\leq 1$ everywhere. Then we can estimate the gradient of $\eta_{ij}$
\[ \begin{split}
     |\nabla \eta_{ij}|= & \left| \frac{\nabla\psi_{ij}\sum_{I(ij)} \psi_{pq}- \psi_{ij} \sum_{I(ij)} \nabla \psi_{pq}}{\left(\sum \psi_{ij}\right)^2} \right| \\
                    \leq & 1+ \sharp I(ij),
   \end{split}
\]
where $I(ij)=\{(pq)|U_{pq}\cap U_{ij}\neq \emptyset\}$. It's clear to see
\[ N_i:= \sup_j\sharp I(ij) \leq l_{i-2}+l_{i-1}+l_i+l_{i+1}+l_{i+2}.
\]

Let
\[ \rho(r)=\frac{V_\lambda(p,r)^{\frac{2}{n-2}}}{r^{\frac{2n}{n-2}}},
\]
and let $\alpha$ be a number to be determined. For any $\phi \in C_0^1(M)$,
\[ \begin{split}
          & \int_M |\phi|^\frac{2n}{n-2}\rho(r)e^{-\alpha r} d\lambda \\
         =& \int_M |\sum \eta_{ij}\phi|^\frac{2n}{n-2}\rho(r)e^{-\alpha r} d\lambda\\
     \leq & \sum_{i=0}^\infty N_i^\frac{n+2}{n-2} \sum_{j=0}^{l_i} \int_{U_{ij}} |\eta_{ij}\phi|^\frac{2n}{n-2}\rho(r)e^{-\alpha r} d\lambda\\
     \leq & \sum_{i=0}^\infty N_i^\frac{n+2}{n-2} \rho(R_{i-1})e^{\frac{2c_0}{n-2}R_{i+2}-\alpha R_{i-1}} \sum_{j=0}^{l_i} \int_{U_{ij}} |\eta_{ij}\phi|^\frac{2n}{n-2} d\lambda.\\
   \end{split}
\]
By Lemma \ref{DSI on annuli}, and the gradient estimate for the partition of unity, we get
\[\int_{U_{ij}} |\eta_{ij}\phi|^\frac{2n}{n-2} d\lambda \leq \left( C_1 e^{C_2 R_{i+2}} \rho(R_{i+2})^{-\frac{n-2}{n}} \int_{(U_{ij})_\delta} |\nabla (\eta_{ij}\phi)|^2 d\lambda \right)^{\frac{n}{n-2}},
\]
and
\[ \int_{U_{ij}} |\nabla (\eta_{ij}\phi)|^2 d\lambda \leq N_i \int_{(U_{ij})_\delta} \phi^2 d\lambda + \int_{(U_{ij})_\delta} \eta_{ij}^2|\nabla \phi|^2 d\lambda.
\]
Hence using the estimates for $N_i$ we derive
\[  \int_M |\phi|^\frac{2n}{n-2}\rho(r)e^{-\alpha r} d\lambda \leq \sum_{i=0}^\infty C(n,C_1)e^{(C(n,c_0,C_2)-\alpha)R_{i-1}} \left(\int_{(A_i)_\delta} \phi^2 + |\nabla \phi|^2 d\lambda \right)^\frac{n}{n-2}.
\]

Let $\alpha > C(n,c_0,C_2)$ and fix $\delta << 1$, we get
\[ \left(\int_M |\phi|^\frac{2n}{n-2}\rho(r)e^{-\alpha r} d\lambda\right)^\frac{n-2}{n}\leq 2C(n,C_1) \int_M \phi^2 + |\nabla \phi|^2 d\lambda.
\]
\end{proof}

If in addition, there is another measure $d\mu$ on $M$, such that $(M,g)$ satisfies a weighted Poincare inequality:
\[ \int_M \phi^2 d\lambda \leq C_P \int_M |\nabla \phi|^2 d\mu.
\]
Then there is an Euclidean type weighted sobolev inequality on M:
\[ \left(\int_M |\phi|^\frac{2n}{n-2}\rho(r)e^{-\alpha r} d\lambda\right)^\frac{n-2}{n}\leq C_S(1+C_P) \int_M |\nabla \phi|^2 (1+\frac{d\lambda}{d\mu})d\mu.
\]

\section{On nontrivial gradient steay Ricci solitons}
\label{steady ricci soliton}

Let's first recall some facts about gradient steady Ricci solitons. We say a Ricci soliton $(M,g,f)$ is nontrivial if the potential function $f$ is nonconstant. Let $R$ be the scalar curvature, it is well-known that on a gradient steady Ricci Soliton, $R+|\nabla f|^2$ is a constant (\cite{Hamilton1993}), which is non-zero if and only if $f$ is nonconstant. We can normalize the metric so that
 \begin{equation}\label{normalization}
   R+|\nabla f|^2=1.
 \end{equation}
 This equation implies that $f$ has at most linear growth. By possibly adding a constant we can let
\[|f(x)|\leq r_p(x)\]
for some $p\in M$.

Denote $dv_f=e^{-f}dv$ where $dv$ is the Riemannian volume form. Since a gradient steady Ricci soliton has nonnegative $\infty-$Bakry-Emery Ricci tensor
\[ Ric+Hess(f)=0,
\]
the result in \cite{Wei-Wylie2009} implies that $(M,g,dv_f)$ satisfies the exponential Jacobian comparison property (\ref{jacobian comparison}) in section \ref{weightedSI}, with $c_0=12$.

Moreover, Munteanu and Wang proved in \cite{Munteanu-Wang2011} that gradient steady Ricci solitons with $R+|\nabla f|^2=\Lambda>0$ have positive bottom f-spectrum $\Lambda/4$, i.e.
\[ \inf_{\phi \in C^1_0(M)} \frac{\int_M |\nabla \phi|^2 dv_f}{\int_M \phi^2 dv_f} = \frac{\Lambda}{4}.
\]

Therefore we can use results in section \ref{weightedSI} to prove
\begin{theorem}\label{weighted SI on Steady Soliton}
Let $(M,g,f)$ be a nontrivial gradient steady Ricci soliton with dimension $n\geq 3$, suppose it is normalized that $|f(x)|\leq r_p(x)$ for some $p\in M$, then for any $\phi \in C^1_0(M)$,
\[\left(\int_M |\phi|^\frac{2n}{n-2} \left(\frac{V_f(p,r)}{r^n}\right)^\frac{2}{n-2} e^{-\alpha r} dv_f\right)^\frac{n-2}{n} \leq C_S \int_M |\nabla \phi|^2 dv_f,
\]
where $\alpha$ and $C_S$ are constants depending only on $n$, and $V_f(p,r)=\int_{B_p(r)} dv_f$.
\end{theorem}

Now recall that the Riemann curvature tensor of a gradient steady Ricci soliton satisfies an elliptic partial differential system:
\begin{equation}\label{eqn for Rm on soliton} -\Delta_f Rm= Q(Rm),
\end{equation}
where $\Delta_f=\Delta - \langle \nabla f, \cdot\rangle$ is the self-adjoint Laplacian with respect to $dv_f$, $Q(Rm)$ is a quadratic term in $Rm$.

\begin{proof}[Proof of Theorem \ref{gapsoliton}]
We can scale the metric to let $\Lambda=1$, and modify $f$ by adding a constant so that $f(p)=0$.
Equation (\ref{eqn for Rm on soliton}) implies that
\[ -\Delta_f |Rm| \leq C(n) |Rm|^2.
\]
For each $R>1$, choose a radial cut-off function $\psi$ such that $\psi(r)=1$ when $0\leq r\leq R$, $\psi(r)=0$ when $r>2R$, and $|\nabla \psi|\leq 1$. Let $u=|Rm|^\frac{n}{4}$, then integrate by parts to get
\[ \begin{split}
     & \int_M \psi^2 |\nabla u|^2 dv_f \\
     = & \int_M-2\psi u \langle \nabla \psi, \nabla u\rangle -\psi^2 u \Delta_f u dv_f \\
     = & \int_M -2\psi u \langle \nabla \psi, \nabla u\rangle -\frac{n}{4}|Rm|^{\frac{n}{4}-1}\psi^2 u \Delta_f|Rm| -(1-\frac{4}{n})\psi^2 |\nabla u|^2 dv_f \\
     \leq & C(n)\int_M \psi^2 u^2 |Rm|dv_f + (\frac{4}{n}-\frac{1}{2})\int_M\psi^2|\nabla u|^2dv_f+2\int_M|\nabla \psi|^2u^2dv_f.
   \end{split}
\]
The assumption $n\geq 3$ implies $\frac{4}{n}-\frac{1}{2}\leq \frac{5}{6}$, hence
\[ \frac{1}{6}\int_M \psi^2 |\nabla u|^2 dv_f\leq  C(n)\int_M \psi^2 u^2 |Rm|dv_f+2\int_M|\nabla \psi|^2u^2dv_f.
\]
Theorem \ref{weighted SI on Steady Soliton} implies
\[ \begin{split}
          & \left( \int_M (\psi u)^\frac{2n}{n-2} \left(\frac{V_f(r)}{r^n}\right)^\frac{2}{n-2} e^{-\alpha r}dv_f\right)^\frac{n-2}{n} \\
     \leq & C_S \int_M |\nabla(\psi u)|^2 dv_f\\
      =   & 2C_S \int_M \psi^2|\nabla u|^2+ |\nabla \psi|^2u^2 dv_f\\
     \leq & C_S C(n)\int_M \psi^2 u^2|Rm| dv_f + 2C_S \int_M |\nabla \psi|^2u^2 dv_f\\
     \leq & C_S C(n) \left(\int_M |Rm|^{\frac{n}{2}}\frac{r^n}{V_f(r)}e^{\frac{n-2}{2}\alpha r} dv_f\right)^\frac{2}{n} \left( \int_M (\psi u)^\frac{2n}{n-2} \left(\frac{V_f(r)}{r^n}\right)^\frac{2}{n-2} e^{-\alpha r}dv_f\right)^\frac{n-2}{n} \\
          & + 2C_S \int_M |\nabla \psi|^2u^2 dv_f.
   \end{split}
\]
Suppose
\[ \left(\int_M |Rm|^{\frac{n}{2}}\frac{r^n}{V_f(r)}e^{\frac{n-2}{2}\alpha r} dv_f\right)^\frac{2}{n}\leq \epsilon < \frac{1}{C(n)C_S},
\]
then
\[\left( \int_M (\psi u)^\frac{2n}{n-2} \left(\frac{V_f(r)}{r^n}\right)^\frac{2}{n-2} e^{-\alpha r}dv_f\right)^\frac{n-2}{n} \leq \frac{2C_S}{1-C_SC(n)\epsilon} \int_{A(R,2R)}u^2 dv_f.
\]
Observe that
\[ \frac{1}{V_f(1)}\int_{A(R,2R)} u^2 dv_f \leq \int_{A(R,2R)} |Rm|^{\frac{n}{2}}\frac{r^n}{V_f(r)}e^{(8+\frac{n-2}{2}\alpha) r} dv_f.\]

Therefore if we take $c= 8+\frac{n-2}{2}\alpha$, then
\[ \int_{A(R,2R)}u^2 dv_f \to 0, \text{ as } R\to \infty,
\]
thus $u\equiv 0$ and $(M,g)$ has to be flat. The universal cover of a flat manifold is $\mathbb{R}^n$ with the Euclidean metric. Since $\nabla\nabla f\equiv 0$ and $|\nabla f|\equiv 1$, the pullback of $f$ to the covering  map will be a linear function.
\end{proof}

\begin{proof}[Proof of theorem \ref{gap by poincare}]
For any $R>0$, let $\psi$ be a cut-off function such that $\psi=1$ on $B_p(R)$, $\psi=0$ on $M\backslash B_p(2R)$ and $|\nabla \psi| \leq \frac{2}{R}$. Let $u=|Rm|^\alpha$. Apply the Poincare inequality by Munteanu-Wang to get
\[ \int_M \psi^2u^2dv_f \leq \frac{4}{\Lambda} \int_M |\nabla(\psi u)|^2 dv_f.
\]
Integration by parts yields
\[  \begin{split}\int_M \psi^2u^2dv_f & \leq \frac{4}{\Lambda} \left( \int_M |\nabla \psi|^2 u^2 dv_f -\int_M \psi^2 u \Delta_f u dv_f \right)\\
                                     & \leq  \frac{4}{\Lambda} \left( \frac{4}{R^2}\int_{B_p(2R)\backslash B_p(R)}  |Rm|^{2\alpha} dv_f -\int_M \psi^2 u \Delta_f u dv_f \right).
     \end{split}
\]
Since $\alpha\geq 1$, we can compute
\[ -\Delta_f u \leq \alpha C(n)|Rm| u.\]
Let $R\to \infty$, we get
\[ \int_M u^2 dv_f\leq \frac{4}{\Lambda} \alpha C(n)\int_M |Rm|u^2dv_f.
\]
The condition $\sup_M |Rm|< \frac{\Lambda}{4\alpha C(n)}$ will force $u$ to be identically $0$.
\end{proof}

\section{On Ricci-flat manifolds}
\label{Ricci-flat manifolds}

Before turning our attention to Ricci-flat manifolds, let's first define $(M,g)$ to be a complete Riemannian manifold of dimension $n$ with $Ric(g)\geq 0$. For simplicity of statements, let's fix an arbitrary point $p\in M$, denote the geodesic ball centered at $p$ with radius $r$ as $B(r)$, and denote its volume as $V(r)$. Integrations in this section are all with respect to the Riemannian volume form. It is well-known that $(M,g)$ has infinite volume, and the volume ratio $\frac{V(r)}{r^n}$ is nonincreasing.

Poincare inequalities on geodesic balls have been proved by P. Li and R. Schoen (\cite{Li-Schoen1984}).
For any $r>0$ and $\phi\in C_0^\infty(B(r))$,
\begin{equation}\label{DPI on balls}
\int_{B(r)} \phi^2\leq C_P r^2 \int_{B(r)}|\nabla \phi|^2,
\end{equation}
where $C_P$ depends only on $n$.
By a well-known result of Saloff-Coste (\cite{Saloff-Coste1992}) and (\ref{DPI on balls}), there are Euclidean type Sobolev inequalities on geodesic balls:
For any $\phi \in C_0^\infty(B(r))$,
\begin{equation}\label{DSI on balls}
\left(\int_{B(r)}\phi^\frac{2n}{n-2} \right)^\frac{n-2}{n}\leq C_S \frac{r^2}{V(r)^\frac{2}{n}}\int_{B(r)} |\nabla \phi|^2,
\end{equation}
where $C_S$ depends only on $n$.

Let $u$ and $h$ be nonnegative functions on a complete Riemannian manifold $(M,g)$ with nonnegative Ricci curvature, suppose $u$ satisfies almost everywhere a partial differential inequality
\[  -\Delta u \leq C_0 hu.
\]
We can apply (\ref{DSI on balls}) to prove the following vanishing theorem for $u$:
\begin{theorem}\label{vanishing theorem by local sobolev}
Suppose $|u|_{L^\frac{2n}{n-2}} < \infty$. Let
\[R_0=\inf\{R| |u|_{L^\frac{2n}{n-2}(B(R))}^2\geq 32C_S|u|_{L^\frac{2n}{n-2}(M\backslash (B(R)))}^2\}.
\]
If
\[\int_M |h|^\frac{n}{2}<\frac{1}{C_0^\frac{n}{2}(\frac{1}{8}+4C_S)^\frac{n}{2}}\frac{V(2R_0)}{(2R_0)^n},\]
then $u\equiv 0$.
\end{theorem}

\begin{proof}
For any $R>0$, choose a cut-off function $\psi\leq 1$ supported on $B(2R)$ with $\psi =1$ on $B(R)$ and$|\nabla \psi|\leq \frac{2}{R}$ . Integration by parts yields
\[ \begin{split}
      \int_M |\nabla (\psi u)|^2 = & \int_{B(2R)\backslash B(R)} |\nabla \psi|^2 u^2 - \int_M \psi^2 u\Delta u \\
                                                       \leq &  \int_{B(2R)\backslash B(R)} |\nabla \psi|^2 u^2 +C_0\int_{B(2R)} hu^2 \\
                                                       \leq &  \frac{4(V(2R)-V(R))^\frac{2}{n}}{R^2}\left(\int_{B(2R)\backslash B(R)} u^\frac{2n}{n-2} \right)^\frac{n-2}{n} +C_0\int_{B(2R)} hu^2.
   \end{split}
\]
By letting $R\to \infty$, we get
\[ \int_M |\nabla u|^2 \leq C_0 \int_M hu^2,
\]
when the righthand side is integrable.

Now let $\phi$ be a cut-off function supported on $B(2R_0)$, such that $\phi =1$ on $B(R_0)$, $|\nabla \phi|\leq \frac{\sqrt{2}}{R_0}$. Apply the  Sobolev inequality (\ref{DSI on balls}) to get
\[ \begin{split}
      \left(\int_{B(R_0)}u^\frac{2n}{n-2} \right)^\frac{n-2}{n}\leq & \left( \int_{B(2R_0)}(\phi u)^\frac{2n}{n-2} \right)^\frac{n-2}{n} \\
                                                              \leq & C_S \frac{4R_0^2}{V(2R_0)^\frac{2}{n}}\int_{B(2R_0)} |\nabla (\phi u)|^2 \\
                                                              \leq & C_S \frac{4R_0^2}{V(2R_0)^\frac{2}{n}} \int_{B(2R_0)} 2|\nabla \phi|^2 u^2 + 2\phi^2|\nabla u|^2 \\
                                                              \leq & C_S \frac{8R_0^2}{V(2R_0)^\frac{2}{n}} \left(\frac{2}{R_0^2}\int_{B(2R_0)\backslash B(R_0)} u^2 + \int_{B(2R_0)}|\nabla u|^2\right)  \\
                                                              \leq & C_S \frac{8R_0^2}{V(2R_0)^\frac{2}{n}} \left(\frac{2}{R_0^2}V(2R_0)^\frac{2}{n}\left(\int_{B(2R_0)\backslash B(R_0)} u^\frac{2n}{n-2} \right)^\frac{n-2}{n} + \int_{B(2R_0)}|\nabla u|^2\right)
   \end{split}
\]
By the assumption, we get
\[ \left(\int_{B(R_0)}u^\frac{2n}{n-2} \right)^\frac{n-2}{n} \leq \frac{16C_S R_0^2}{V(2R_0)^\frac{2}{n}} \int_{B(2R_0)} |\nabla u|^2.
\]
Hence
\[ \begin{split}
    \int_M |\nabla u|^2 \leq & C_0 \left( \int_M |h|^\frac{n}{2}\right)^\frac{2}{n} \left[ \left(\int_{B(R_0)} u^\frac{2n}{n-2}\right)^\frac{n-2}{n}+ \left( \int_{M\backslash B(R_0)} u^\frac{2n}{n-2} \right)^\frac{n-2}{n} \right]  \\
                                          \leq & C_0 (1+ \frac{1}{32C_S})\left( \int_M |h|^\frac{n}{2}\right)^\frac{2}{n} \left( \int_{B(R_0)} u^\frac{2n}{n-2}\right)^\frac{n-2}{n} \\
                                          \leq &  C_0 (1+ \frac{1}{32C_S}) \frac{16C_S R_0^2}{V(2R_0)^\frac{2}{n}}\left( \int_M |h|^\frac{n}{2}\right)^\frac{2}{n} \int_{B(2R_0)} |\nabla u|^2.
    \end{split}
\]
Under the assumption on $h$, we get $\nabla u \equiv 0$, hence $u\equiv 0$ given its integrability.
\end{proof}

Similarly we can apply (\ref{DPI on balls}) to prove:

\begin{theorem}\label{vanishing theorem by local poincare}
Suppose $u\in L^{2}$ . Let
\[ R_0= \inf\{R| |u|_{L^2(B(R))}^2 \geq 32C_P |u|_{L^2(M\backslash B(R))}^2\}.
\]
If
\[ \sup_M |h| < C_0^{-1}(\frac{1}{2}+ 16C_P)^{-1}R_0^{-2},
\]
then $u\equiv 0$.
\end{theorem}

From now on we will focus on Ricci-flat manifolds. The Riemann curvature tensor of a Ricci-flat manifold satisfies an elliptic system:
\[ -\Delta Rm= Q(Rm),\]
where $Q(Rm)$ is a quadratic term of $0^\text{th}$ order. Moreover, there is a refined Kato's inequality (\cite{Bando-Kasue-Nakajima1989}\cite{Calderbank-Gauduchon-Herzlich2000}):

\[|\nabla |Rm||^2 \leq \frac{n-1}{n+1} |\nabla Rm|^2,
\]
where $n\geq 4$ is the dimension. This inequality leads to the following lemma:

\begin{lemma}\label{refined Kato inequality}
Let $(M,g)$ be a Ricci-flat Riemannian manifold of dimension $n\geq 4$, then for any $\alpha \geq \frac{n-3}{n-1}$,
\begin{equation}\label{refined pde for Rm}
-\Delta |Rm|^\alpha \leq C(n)\alpha |Rm|^{\alpha+1},
\end{equation}
where $C(n)$ is a constant depending only on $n$.
\end{lemma}

\begin{proof}
A detailed proof can be found in \cite{Minerbe2009}.
\end{proof}

With the help of Lemma \ref{refined Kato inequality}, Theorem \ref{vanishing theorem by local sobolev} and \ref{vanishing theorem by local poincare} directly yield the following rigidity results:

\begin{theorem}\label{gap thm by local sobolev}
Let $(M,g)$ be a complete Ricci-flat Riemannian manifold with $|Rm|^\alpha\in L^{\frac{2n}{n-2}}(M)$, for some $\alpha\geq \frac{n-3}{n-1}$. Let
\[R_0=\inf_{p\in M}\{R| |Rm|_{L^\frac{2n\alpha}{n-2}(B_p(R))}^{2\alpha}\geq 32C_S|Rm|_{L^\frac{2n\alpha}{n-2}(M\backslash B_p(R))}^{2\alpha}\}.
\]

If
\[\int_M |Rm|^\frac{n}{2}<\frac{1}{\alpha^\frac{n}{2}C(n)^\frac{n}{2}(\frac{1}{8}+4C_S)^\frac{n}{2}}\frac{V(B_p(2R_0))}{(2R_0)^n},\]
then $(M,g)$ is flat. Here $C(n)$ is the same constant in (\ref{refined pde for Rm}), $C_S$ is the constant in the local Sobolev inequality (\ref{DSI on balls}).
\end{theorem}

\begin{theorem}\label{gap thm by local poincare}
Let $(M,g)$ be a complete Ricci-flat Riemannian manifold with $|Rm|\in L^{2\alpha}(M)$ for some $\alpha \geq \frac{n-3}{n-1}$. Let
\[ R_0= \inf_{p\in M}\{R| |Rm|_{L^{2\alpha}(B_p(R))}^{2\alpha} \geq 32C_P |Rm|_{L^{2\alpha}(M\backslash B_p(R))}^{2\alpha}\}.
\]
If
\[ \sup_M |Rm| < \alpha^{-1}C(n)^{-1}(\frac{1}{2}+ 16C_P)^{-1}R_0^{-2},
\]
 then $(M,g)$ is flat. Here $C(n)$ is the same constant as in (\ref{refined pde for Rm}), $C_P$ is the constant in the local Poincare inequality (\ref{DPI on balls}).
\end{theorem}

\begin{proof}[Proof of Theorem \ref{gap thm by local sobolev} and \ref{gap thm by local poincare}]

By Lemma \ref{refined Kato inequality},we can directly apply Theorem \ref{vanishing theorem by local sobolev} and \ref{vanishing theorem by local poincare} with $u=|Rm|^\alpha$, $h=|Rm|$ to prove Theorem \ref{gap thm by local sobolev} and \ref{gap theorem for ricci-flat manifolds}.
\end{proof}

To prove Theorem \ref{gap theorem for ricci-flat manifolds}, we need to recall the regularity estimates in \cite{Carron2010}, where Carron defined the following:

\begin{definition}\label{gamma k regular}
A Riemannian manifold $(M,g)$ is $(\Gamma,k)$ regular if for any $x\in M$, $r>0$ and $\delta \in (0,1)$ such that
\[\sup_{B_x(\delta r)} |Rm| \leq \frac{1}{r^2},\]
there are estimates for up to $k^\text{th}$ order covariant derivatives of the curvature:
\[ \sup_{B_x(\frac{1}{2}\delta r)} |\nabla^j Rm|\leq \frac{\Gamma}{\delta^j r^{j+2}}, j=1,2,...,k.\]
\end{definition}

\begin{lemma}[G. Carron]\label{regularity of einstein metric}
If $(M,g)$ is Ricci-flat, then $(M,g)$ is $(\Gamma,1)$ regular with $\Gamma$ depending only on the dimension $n$.
\end{lemma}
 \begin{proof}
 This lemma is proved in the more general setting of critical metrics in \cite{Carron2010}. In the special case of Einstein metrics, it can be seen as an elliptic version of Shi's estimate \cite{Shi1989}.  See \cite{Carron2010} for details.
 \end{proof}

Now we can prove Theorem \ref{gap theorem for ricci-flat manifolds}.

\begin{proof}[Proof of Theorem \ref{gap theorem for ricci-flat manifolds}]

When $n\leq 3$, a Ricci-flat manifold is necessarily flat. When $n\geq 4$, observe that $\frac{n-2}{4} >\frac{n-3}{n-1}$, so we can apply Theorem \ref{gap thm by local sobolev} with $\alpha=\frac{n-2}{4}$, provided we can estimate both the smallness and the decay rate of $|Rm|_{L^\frac{n}{2}}$.

To find a region near $p$ with large curvature, we use a point picking process.
Let
\[ r=\sup\{l>0| \sup_{B_p(l)}|Rm|\leq \frac{1}{Ll^2}\},
\]
for some $L\geq 1$ to be determined later.
Then there exists $x_0\in  \overline{B_p(r)}$ such that
\[ |Rm|(x_0)=\frac{1}{Lr^2}.
\]
Let $r_0=\frac{r}{2}$. For each integer $i\geq 0$, if $\sup_{B_{x_i}(\frac{r_i}{2})}|Rm|\leq \frac{4}{Lr_i^2}$, then stop. Otherwise there exists $x_{i+1} \in B_{x_i}(\frac{r_i}{2})$ such that
\[ |Rm|(x_{i+1})=\frac{4}{Lr_i^2}.
\]
Let $r_{i+1}=\frac{r_i}{2}$. If this process doesn't stop in finite steps, then we can find a sequence $x_i,r_i$,$i=0,1,2,...$ such that
\[r_i=\frac{r}{2^{i+1}},\]
\[|Rm|(x_{i+1})=\frac{1}{Lr_i^2}=\frac{4^{i+1}}{Lr^2},\]
\[d(x_0,x_{i+1})\leq \sum_{j=1}^{i+1} r_j  \leq \frac{r}{2},\]
leading to a contradiction since $\sup_{B_{x_0}(r)}|Rm|< \infty$.
Therefore the point picking process must stop at a finite step $N$, denote $\bar{x}=x_N$ and $\bar{r}=r_{N+1}$, then
\[|Rm|(\bar{x})=\frac{1}{4L\bar{r}^2},\]
\[\sup_{B_{\bar{x}}(\bar{r})} \leq \frac{1}{L\bar{r}^2}.\]
By Lemma \ref{regularity of einstein metric},
\[\sup_{B_{\bar{x}}(\frac{\bar{r}}{2})} |\nabla Rm| \leq \frac{\Gamma}{\bar{r}^3}.\]
Hence for any $y\in B_{\bar{x}}(\frac{\bar{r}}{8L\Gamma})$,
\[ |Rm|(y)\geq \frac{1}{4L\bar{r}^2} - \frac{\Gamma}{\bar{r}^3} \frac{\bar{r}}{8L\Gamma}=\frac{1}{8L\bar{r}^2}.
\]
Denote $\tilde{r}=\frac{\bar{r}}{8L\Gamma}$. Since $B_{\bar{x}}(\tilde{r}) \subset B_p(2r)$, we have
\[ \int_{B_p(2r)} |Rm|^\frac{n}{2}\geq \int_{B_{\bar{x}}(\tilde{r})} |Rm|^\frac{n}{2} \geq \frac{1}{(8L)^\frac{3n}{2}\Gamma^n} \rho_{\bar{x}}(\tilde{r}) \geq  \frac{1}{(\frac{7}{2})^n(8L)^\frac{3n}{2}\Gamma^n} \rho_{p}(2r),
\]
the last inequality comes from the Bishop-Gromov volume comparison theorem. Since $\rho_{p}(r)$ is non-increasing, the above inequality implies
\begin{equation}\label{vanishing condition 1}
\begin{split}
\int_{B_p(2r)}|Rm|^\frac{n}{2} \geq & \frac{1}{\epsilon(\frac{7}{2})^n(8L)^\frac{3n}{2}\Gamma^n} \rho_{p}(2r) \int_M |Rm|^\frac{n}{2} \rho_p^{-1} \\
                                \geq & \frac{1}{\epsilon(\frac{7}{2})^n(8L)^\frac{3n}{2}\Gamma^n} \int_{M\backslash B_p(2r)} |Rm|^\frac{n}{2} .
\end{split}
\end{equation}

On the other hand,
\[ \int_{B_p(r)} |Rm|^\frac{n}{2} \leq \left( \frac{1}{Lr^2} \right)^\frac{n}{2} V(B_p(r)) \leq \frac{2^n}{L^\frac{n}{2}} \rho_p(2r),
\]
and
\[ \int_{M\backslash B_p(r)} |Rm|^\frac{n}{2} \leq \rho_p(r) \int_{M\backslash B_p(r)} |Rm|^\frac{n}{2} \rho_p^{-1} \leq \epsilon 2^n \rho_p(2r),
\]
hence
\begin{equation} \label{vanishing condition 2}
\int_M |Rm|^\frac{n}{2} \leq \left( \frac{2^n}{L^\frac{n}{2}}+\epsilon 2^n \right)2^n \rho_p(4r).
\end{equation}

 If we first choose $L$ large enough, and then choose $\epsilon$ small enough, all depending only on $n$, then (\ref{vanishing condition 1}) and (\ref{vanishing condition 2}) together with the monotonicity of $\rho_p(r)$ imply that the conditions in Theorem \ref{gap thm by local sobolev} are satisfied, in particular $R_0\leq 2r$. Therefore $Rm\equiv 0$.
\end{proof}

\bibliographystyle{plain}
\bibliography{ref}

\end{document}